\documentclass{amsart}
 
\usepackage{amsmath}
\usepackage{amsthm,amssymb,color,comment}
\usepackage{bbm}
\usepackage{hyperref}
\usepackage{url}
\usepackage[TS1,T1]{fontenc}
\usepackage[utf8]{inputenc}
\usepackage{dsfont}
\usepackage{tikz}
\usepackage{enumitem}
\usepackage{subfigure}
\usepackage[nocompress]{cite}

\newtheorem{mainthm}{Theorem}

\numberwithin{equation}{section}

\newtheorem{thm}{Theorem}[section]

\newtheorem{lem}[thm]{Lemma}

\newtheorem{Def}[thm]{Definition}

\theoremstyle{definition}
\newtheorem{Ass}[thm]{Assumption}
\newtheorem{rem}[thm]{Remark}
\newtheorem{cor}[thm]{Corollary}
\newtheorem{Nott}[thm]{Notation}

\DeclareMathOperator{\suppo}{supp}

\DeclareMathOperator*{\wstlim}{w^*-lim}

\newcommand{\R}{\mathbb{R}}
\newcommand{\N}{\mathbb{N}}

\newcommand{\diff}{\mathop{}\!\mathrm{d}}
\newcommand{\wstar}{\overset{\ast}{\rightharpoonup}}

\newcommand{\ueps}{u^{\varepsilon}}
\newcommand{\veps}{v^{\varepsilon}}

\newcommand{\weaks}{\overset{\ast}{\rightharpoonup}}

\newcommand{\weak}{\rightharpoonup}

\makeatletter
\newcommand{\doublewidetilde}[1]{{%
  \mathpalette\double@widetilde{#1}%
}}
\newcommand{\double@widetilde}[2]{%
  \sbox\z@{$\m@th#1\widetilde{#2}$}%
  \ht\z@=.9\ht\z@
  \widetilde{\box\z@}%
}
\makeatother
\author{Jakub Skrzeczkowski}
\address{Faculty of Mathematics, Informatics and Mechanics, University of Warsaw, Stefana Banacha 2, 02-097 Warsaw, Poland}
\email{jakub.skrzeczkowski@student.uw.edu.pl}
\thanks{Jakub Skrzeczkowski was supported by National Science Center, Poland through project no. 2017/27/B/ST1/01569. He is grateful to Benoît Perthame for fruitful discussions and helpful suggestions.}

\begin{document}

\title[Fast reaction limit and forward-backward diffusion]{Fast reaction limit and forward-backward diffusion:\\ A Radon-Nikodym approach}

\begin{abstract}
We consider two singular limits: fast reaction limit with nonmonotone nonlinearity and regularization of forward-backward diffusion equation. It was proved by Plotnikov that for cubic-type (nondegenerate) nonlinearities, the limit oscillates between at most three states. In this paper we make his argument more optimal and we sharpen the previous result: we use Radon-Nikodym theorem to obtain a pointwise identity characterizing the Young measure. As a consequence, we establish a simpler condition which implies Plotnikov result for piecewise affine functions. We also prove that the result is true if the Young measure is not supported in the so-called unstable zone, the fact observed in numerical simulations.
\end{abstract}

\keywords{reaction-diffusion, fast reaction limit, forward-backward diffusion, Young measures, cross-diffusion, oscillations, compensated compactness}

\subjclass{35K57, 35B25, 35B36}

\maketitle
\setcounter{tocdepth}{1}
\tableofcontents

\section{Introduction and main results}
\subsection{Presentation of the problem} In this paper we focus on two interesting limit problems: fast-reaction limit in the reaction-diffusion system
\begin{equation}\label{eq:fastreact}
    \begin{split}
        \partial_t \ueps &= \frac{F(\ueps) - \veps}{\varepsilon},\\
        \partial_t \veps &= \Delta \veps + \frac{\veps - F(\ueps)}{\varepsilon}
    \end{split}
\end{equation}
and regularization of the forward-backward parabolic equation $\partial_t u = \Delta F(u)$
\begin{equation}\label{eq:forback}
    \begin{split}
        \partial_t \ueps &= \Delta \veps,\\
        \veps &= F(\ueps) + \varepsilon\, \partial_t \ueps,
    \end{split}
\end{equation}
where $F$ is a nonmonotone function, for simplicity assumed to look like as in Fig. \ref{plot:Fd}. Notice that due to nonmonotone character of $F$, it has three inverses $S_1$, $S_2$ and $S_3$. Both problems are posed on some bounded domain $\Omega \subset \R^d$ and are equipped with initial conditions and usual Neumann boundary conditions.\\

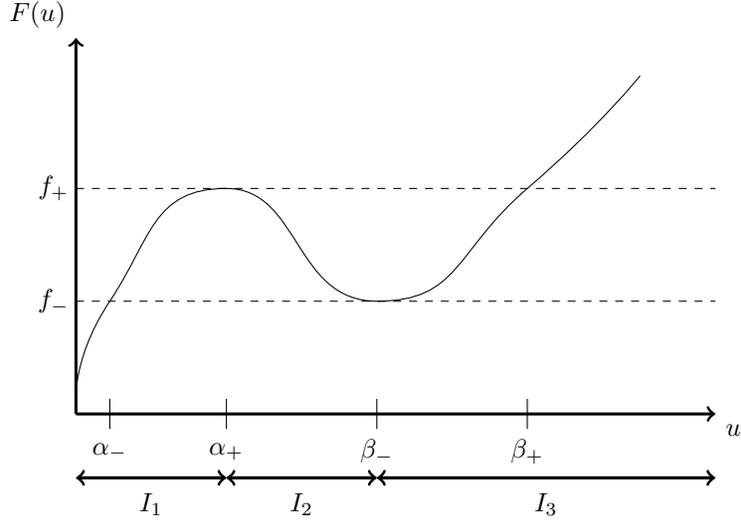
\begin{figure}
\begin{tikzpicture}

\draw[line width=0.4mm,->] (0,0) -- (8.5,0) node[anchor=north west] {$u$};
\draw[line width=0.4mm,->] (0,0) -- (0,5) node[anchor=south east] {$F(u)$};

\draw [black] plot [smooth, tension=1] coordinates { (0,0.0) (0.45,1.5) (2,3.0) (4,1.5) (6,3.0) (7.5, 4.5)};

\draw [dashed] (0,3) -- (8.5,3);
\draw [dashed] (0,1.5) -- (8.5,1.5);
\node at (-0.3, 3) {$f_{+}$};
\node at (-0.3, 1.5) {$f_{-}$};

\draw[line width=1.2pt, <->] (0,-0.85) -- (2,-0.85);
\node at (1,-1.2) {$I_{1}$};
\draw[line width=1.2pt, <->] (2,-0.85) -- (4,-0.85);
\node at (3,-1.2) {$I_{2}$};
\draw[line width=1.2pt, <->] (4,-0.85) -- (8.5,-0.85);
\node at (6.25,-1.2) {$I_{3}$};

\node at (0.45,-0.5) {$\alpha_{-}$};
\node at (2,-0.5) {$\alpha_{+}$};
\node at (4,-0.5) {$\beta_{-}$};
\node at (6,-0.5) {$\beta_{+}$};
\draw (0.45,-0.2) -- (0.45,0.2);
\draw (2,-0.2) -- (2,0.2);
\draw (4,-0.2) -- (4,0.2);
\draw (6,-0.2) -- (6,0.2);

\end{tikzpicture}
\vspace{-4mm}
\caption{Plot of a typical function $F$. It is strictly increasing in the intervals $I_1:= (-\infty, \alpha_{+}]$, $I_3:= [\beta_{-}, \infty)$ and strictly decreasing in $I_2:= (\alpha_{+}, \beta_{-})$. For $r \in [f_{-}, f_{+}]$, the function $F$ is not invertible and equation $F(u) = r$ has three roots $u = S_1(r)\leq S_2(r) \leq S_3(r)$.}
\label{plot:Fd}
\end{figure}

\noindent System \eqref{eq:fastreact} is an interesting toy model for studying oscillations in reaction-diffusion systems as they are known to occur in its steady states \cite{MR3908864}. For monotonone $F$ the problem is fairly classical and has been studied for a great variety of reaction-diffusion systems, also with more than two components \cite{MR2009623, MR3005532, MR4040718,MR2776460} or reaction-diffusion equation coupled with an ODE \cite{MR3655798}. In the limit $\varepsilon \to 0$ one obtains widely studied cross-diffusion systems \cite{MR4119040,MR4072681,MR3359162,MR3165911,MR3350617,MR1415047} where the gradient of one quantity induces a flux of another one. A slightly different yet connected type of problem deals with the fast-reaction limit for irreversible reactions which leads to free boundary problems \cite{MR3501846,MR589954,MR2562164}. Finally, for nonmonotone $F$ as in this paper, the only available result was established very recently in \cite{perthame2020fast} (see below). We also refer to the recent stability analysis of problems of the type \eqref{eq:fastreact} \cite{MR3039206,cygan2021stability,cygan2021instability}. \\

\noindent System \eqref{eq:forback} was extensively studied by Plotnikov \cite{MR1299852,MR1242166} who identified limits as $\varepsilon \to 0$ in terms of Young measures (see below) and by Novick-Cohen and Pego who studied its asymptotics with $\varepsilon>0$ fixed \cite{MR1015926}. Regularization term in \eqref{eq:forback} was also generalized in \cite{MR3466547, MR3148065,MR3010139}. Recently, so called nonstandard analysis was used to study the limit problem in the space of grid functions \cite{MR3899968, MR4151175}.   \\

\noindent It is known \cite{MR1299852,perthame2020fast} that both systems exhibit the following surprising phenomenon: as $\varepsilon \to 0$, $F(\ueps)  \to v$ and $\veps \to v$ converge strongly without any known a priori estimates allowing to conclude so. As a consequence, $u^{\varepsilon}$ converges weakly to
$$
u(t,x) = \lambda_1(t,x)\, S_1(v(t,x)) + \lambda_2(t,x)\, S_2(v(t,x)) + \lambda_3(t,x)\,S_3(v(t,x))
$$
where $\sum_{i=1}^3 \lambda_i(t,x) = 1$. More precisely, if $\mu_{t,x}$ is a Young measure generated by $\{\ueps\}_{\varepsilon > 0}$ we have
$$
\mu_{t,x} = \lambda_1(t,x)\, \delta_{S_1(v(t,x))} + \lambda_2(t,x)\, \delta_{S_2(v(t,x))} + \lambda_3(t,x)\, \delta_{S_3(v(t,x))}
$$
which represents oscillations between phases $S_1(v(t,x))$, $S_2(v(t,x))$ and $S_3(v(t,x))$. The proof exploits a family of energies as well as analysis of related Young measures in the spirit of Murat and Tartar work on conservation laws and compensated compactness \cite{MR894077,MR584398}. The numerical simulations suggests that the middle state, referred to as an unstable phase, is not present \cite{MR2103092} which motivates research on two-phase solutions to such problems \cite{MR2563628, MR2644916, MR2765690,MR2914664} with a result of nonuniqueness when the unstable phase is present \cite{MR3159815}.\\

\noindent So far, the main assumption on $F$ that allows to deduce strong convergence is the so-called nondegeneracy condition: for \eqref{eq:fastreact} it reads
\begin{equation}\label{eq:nondegeneracy:fast_reaction}
\mbox{for all intervals } R \subset (f_-,f_+): \quad \sum_{i=1}^3 a_i\;  \big( S_1'(r)+1\big) =0 \mbox{ for } r\in R \implies a_1 +a_2+a_3=0
\end{equation}
while for \eqref{eq:forback} it reads
\begin{equation}\label{eq:nondegeneracy:forback}
\mbox{for all intervals } R \subset (f_-,f_+): \quad \sum_{i=1}^3 a_i\;  S_1'(r) =0 \mbox{ for } r\in R \implies a_1 +a_2+a_3=0.
\end{equation}
While it is fairly classical for this type of problems \cite{MR657784, MR1015926,MR1299852}, it is hard to be verified for a given nonlinearity $F$. Moreover, the nondegeneracy condition excludes piecewise affine functions which allows for more explicit computations as in \cite{MR2563628}.\\

\subsection{Main results and outline of the paper} In this paper, we take a slightly different approach to study strong convergence. Although we use family of energy identities to characterize Young measure as Plotnikov \cite{MR1299852}, we aim at pointwise identities to obtain optimal amount of information from these energy identities, in particular new results. To achieve this, we use Radon-Nikodym Theorem as explained below. \\

\noindent Let $\{\mu_{t,x}\}_{t,x}$ be Young measure generated by sequence $\{\ueps\}_{\varepsilon \in (0,1)}$ solving either \eqref{eq:fastreact} or \eqref{eq:forback}, i.e. for any bounded function $G: \R \to \R$ we have (up to a subsequence and for a.e. $(t,x) \in (0,T)\times \Omega$)
$$
G(\ueps) \wstar \int_{\R} G(\lambda) \diff \mu_{t,x}(\lambda),
$$
see Appendix \ref{app:YM} if necessary. To analyze amount of $\mu_{t,x}$ on the intervals $I_1$, $I_2$ and $I_3$, see Figure \ref{plot:Fd}, we introduce restrictions
$$
\mu_{t,x}^{(1)} := \mu_{t,x} \, \mathds{1}_{I_{1}}, \qquad \qquad \mu_{t,x}^{(2)} := \mu_{t,x} \, \mathds{1}_{I_{2}}, \qquad \qquad \mu_{t,x}^{(3)} := \mu_{t,x} \, \mathds{1}_{I_{3}}.
$$
The reason we introduce these measures is that in the sequel, we will gain information only about measure $F^{\#}\mu_{t,x}$, i.e. push-forward (image) of $\mu_{t,x}$ along $F$ defined as
$$
F^{\#}\mu_{t,x} = \mu_{t,x}(F^{-1}(A)), \qquad A \subset \R^+. 
$$
Observe that for all $i=1,2,3$ measures $F^{\#}\mu_{t,x}^{(i)}$ are absolutely continuous with respect to $F^{\#}\mu_{t,x}$. Therefore, Radon-Nikodym theorem implies that there exist densities $g^{(1)}(\lambda)$, $g^{(2)}(\lambda)$ and $g^{(3)}(\lambda)$ such that 
\begin{equation}\label{eq:def_density_gi}
F^{\#}\mu_{t,x}^{(i)}(A) = \int_{A} g^{(i)}(\lambda) \diff F^{\#}\mu_{t,x}(\lambda), \qquad \qquad i = 1, 2, 3.
\end{equation}
We also note that for all $A \subset \R^+$
\begin{equation}\label{eq:summabilityFmu}
\sum_{i=1}^3 F^{\#}\mu_{t,x}^{(i)}(A) = 
\sum_{i=1}^3 \mu_{t,x}(F^{-1}(A) \cap I_i) = \mu_{t,x}(F^{-1}(A)) = F^{\#}\mu_{t,x}(A).
\end{equation}
In particular, from \eqref{eq:def_density_gi} and \eqref{eq:summabilityFmu} we deduce that for $F^{\#}\mu_{t,x}$-a.e. $\lambda$ we have
\begin{equation}\label{eq:summabilityg}
\sum_{i=1}^3 g_i(\lambda) = 1.
\end{equation}

\noindent The main result of this paper reads:
\begin{mainthm}\label{mainthm:simple_entropy_equality}
\text{\bf (A)} Let $\{\mu_{t,x}\}_{t,x}$ be Young measure generated by sequence $\{\ueps\}_{\varepsilon \in (0,1)}$ solving \eqref{eq:fastreact}. Then, for almost all $\lambda_0$ (with respect to $F^{\#}\mu_{t,x}$) and all $\tau_0 \neq f_{-}, f_{+}$ we have
$$
\sum_{i=1}^3 (S_i'(\tau_0) + 1) \, \left[\mathds{1}_{\lambda_0 > \tau_0} \, g_i(\lambda_0) - F^{\#}\mu_{t,x}^{(i)}(\tau_0, \infty) \right] + (S_1'(\tau_0) - S_2'(\tau_0)) \, (F^{\#}\mu_{t,x}^{(1)}(\R^+) - g_1(\lambda_0))= 0.
$$
where $S_i$ are inverses of $F$ as in Notation \ref{intro:not_inv_phi} and $g_i$ are Radon-Nikodym densities as in \eqref{eq:def_density_gi}. Moreover, if $\lambda_0 \neq f_{-}, f_{+}$ it holds
\begin{equation}\label{eq:localized_final_entropy}
\left(1 - F^{\#}\mu_{t,x}\left\{\lambda_0\right\} \right) \, \sum_{i=1}^{3} (S_i'(\lambda_0) + 1) \, g_i(\lambda_0) \,  = 0.
\end{equation}
{\bf (B)} If $\{\mu_{t,x}\}_{t,x}$ is the Young measure generated by sequence $\{\ueps\}_{\varepsilon \in (0,1)}$ solving \eqref{eq:forback} the equalities above holds with functions $S_i'$ instead of $S_i' + 1$.
\end{mainthm}
\noindent As $F^{\#}\mu_{t,x}$ turns out to be the Young measure generated by $\{\veps\}_{\varepsilon > 0}$ cf. Corollary \ref{cor:connection_YM}, strong convergence $\veps \to v$ follows from proving that $F^{\#}\mu_{t,x}$ is a Dirac mass cf. Lemma \ref{lem:YM_prop} (A). Equation \eqref{eq:localized_final_entropy} shows that it is sufficient to find $\lambda_0$ in the support such that the sum $\sum_{i=1}^{3} (S_i'(\lambda_0) + 1) \, g_i(\lambda_0)$ does not vanish (some additional care is needed when $\lambda_0 = f_-, f_+$, cf. Lemma \ref{lem:case_f-f+}).\\

\noindent We remark that similar forms of entropy equality as in Theorem \ref{mainthm:simple_entropy_equality} are well-known however they are not so easily formulated and they are usually stated without explicitly identified coefficients standing next to $(S_{i}'(\tau_0) + 1)$. First, we show that the form presented in Theorem \ref{mainthm:simple_entropy_equality} can be used to recover already known result due to Plotnikov \cite{MR1299852} as well as Perthame and Skrzeczkowski \cite{perthame2020fast}.
\begin{mainthm}\label{mainthm:oldplotnikov}
Suppose that nondegeneracy condition \eqref{eq:nondegeneracy:fast_reaction} or \eqref{eq:nondegeneracy:forback} is satisfied. Then, $\veps \to v$ strongly in $L^2((0,T)\times \Omega)$. Moreover, there are nonnegative numbers $\lambda_1(t,x)$, $\lambda_2(t,x)$, $\lambda_3(t,x)$ such that $\sum_{i=1}^3 \lambda_i(t,x) = 1$ and
$$
\mu_{t,x} = \lambda_1(t,x)\, \delta_{S_1(v(t,x))} + \lambda_2(t,x)\, \delta_{S_2(v(t,x))} + \lambda_3(t,x)\, \delta_{S_3(v(t,x))}.
$$
\end{mainthm}
\noindent Now, we move to the new results that easily follow from Theorem \ref{mainthm:simple_entropy_equality}. The first one asserts that if one knows a priori that the Young measure $\mu_{t,x}$ is not supported in the interval $I_2$ where $F$ is decreasing, the strong convergence occurs. The fact concerning support of $\mu_{t,x}$ was observed in numerical simulations \cite{MR2103092} and so, the next theorem may serve as a tool to prove strong convergence without nondegeneracy condition.

\begin{mainthm}\label{mainthm:emptyunstablephase}
Suppose that:
\begin{itemize}
    \item there exists $\tau_0 \in (f_-, f_+)$ such that $S_1'(\tau_0) - S_3'(\tau_0) \neq 0$,
    \item the Young measure $\{\mu_{t,x}\}_{t,x}$ is not supported in the interval $I_2$ (see Figure \ref{plot:Fd}).
\end{itemize}
Then, $\veps \to v$ strongly in $L^2((0,T)\times \Omega)$. Moreover, there are nonnegative numbers $\lambda_1(t,x)$, $\lambda_3(t,x)$ such that $\lambda_1(t,x) + \lambda_3(t,x) = 1$ and
$$
\mu_{t,x} = \lambda_1(t,x)\, \delta_{S_1(v(t,x))}  + \lambda_3(t,x)\, \delta_{S_3(v(t,x))}.
$$
\end{mainthm}
\noindent The next result shows that the systems \eqref{eq:fastreact} and \eqref{eq:forback} are not exactly the same in view of the strong convergence. Indeed, for the first one, we can establish a simple condition on $F$ implying strong convergence of $\veps \to v$ that does not exclude piecewise affine functions as in the case of nondegeneracy condition \eqref{eq:nondegeneracy:fast_reaction}.

\begin{mainthm}\label{mainthm:newconditiononF}
Let $\{\mu_{t,x}\}_{t,x}$ be a Young measure generated by sequence $\{\ueps\}_{\varepsilon \in (0,1)}$ solving \eqref{eq:fastreact}. Suppose that:
\begin{itemize}
    \item there exists $\tau_0 \in (f_-, f_+)$ such that $S_1'(\tau_0) - S_3'(\tau_0) \neq 0$,
    \item $
S_2'(\lambda) + 1 > 0$ for all $\lambda \in (f_-, f_+)$.
\end{itemize}
Then, $\veps \to v$ strongly in $L^2((0,T)\times \Omega)$. Moreover, there are nonnegative numbers $\lambda_1(t,x)$, $\lambda_2(t,x)$, $\lambda_3(t,x)$ such that $\sum_{i=1}^3 \lambda_i(t,x) = 1$ and
$$
\mu_{t,x} = \lambda_1(t,x)\, \delta_{S_1(v(t,x))} + \lambda_2(t,x)\, \delta_{S_2(v(t,x))} + \lambda_3(t,x)\, \delta_{S_3(v(t,x))}.
$$
\end{mainthm}

\noindent As an example of function $F$ satisfying assumptions of Theorem \ref{mainthm:newconditiononF} consider
$$
F(\lambda) = \begin{cases}
2\lambda &\mbox{ if } \lambda \in [0,1],\\
3 - 2\lambda &\mbox{ if } \lambda \in \left[1, \frac{5}{4} \right],\\
4\lambda - \frac{9}{2} &\mbox{ if } \lambda \in \left[\frac{5}{4}, \infty \right)
\end{cases}
$$
Then, $S_1'(\lambda) = \frac{1}{2}$, $S_2'(\lambda) = -\frac{1}{2}$ and $S_3'(\lambda) = \frac{1}{4}$ so that $S_1'(\lambda)-S_3'(\lambda) = \frac{1}{4} \neq 0$ and $S_2'(\lambda) + 1 = \frac{1}{2} > 0$. Note that $F$ does not satisfy nondegeneracy condition \eqref{eq:nondegeneracy:fast_reaction} that was used in the previous paper on fast reaction limit with nonmonotone reaction function \cite{perthame2020fast}.\\

\noindent Proofs of Theorem \ref{mainthm:emptyunstablephase} and \ref{mainthm:newconditiononF} are based on equation \eqref{eq:localized_final_entropy}, namely one uses $g_1(\lambda_0) + g_2(\lambda_0) + g_3(\lambda_0) = 1$ to show that for $\lambda_0 \in \suppo F$ we have $F^{\#}\mu_{t,x}\{\lambda_0\} = 1$. Note however that \eqref{eq:localized_final_entropy} is not valid for $\lambda_0 = f_{-}, f_{+}$ so that some additional care is needed if the support of measure $F^{\#}\mu_{t,x}$ accumulates only in these points. This is studied in Lemma \ref{lem:case_f-f+} and it requires an additional assumption that $S_1'(\tau) - S_3'(\tau)$ does not vanish at least for one value of $\tau$, see also Remark \ref{rem:sharper_ass_casef-f+}.\\

\noindent The structure of the paper is as follows. In Section \ref{sect:prop_fast_react} we review (well-known) properties of the fast-reaction system \eqref{eq:fastreact}. Then, in Section \ref{sect:proof_mainmain_thm} we use compensated compactness approach to prove Theorem \ref{mainthm:simple_entropy_equality}. Section \ref{sect:proofs_theorems} is devoted to the simple proofs of Theorems \ref{mainthm:oldplotnikov}, \ref{mainthm:emptyunstablephase} and \ref{mainthm:newconditiononF} while in Section \ref{sect:adaptation} we show how to easily adapt Theorems \ref{mainthm:simple_entropy_equality}--\ref{mainthm:emptyunstablephase} to the case of system \eqref{eq:forback}. Finally, Appendix \ref{app:genlabel} provides necessary background on Young measures, supports of measures and compensated compactness results.

\subsection{Technical assumptions and notation} For the sake of completeness, we list here assumptions of technical nature.
\begin{Nott}\label{intro:not_inv_phi}
Let $S_1(\lambda) \leq S_2(\lambda) \leq S_3(\lambda)$ be the solutions of equation $F(S_i(\lambda)) = \lambda$ as already introduced in Figure \ref{plot:Fd}. These are inverses of $F$ satisfying
$$
S_1:(-\infty,f_{+}] \to (-\infty,\alpha_{+}], \qquad S_2:(f_{-},f_{+}) \to (\alpha_{+},\beta_{-}), \qquad S_3:[f_{-},\infty) \to [\beta_{-},\infty).
$$
Their role is too focus analysis on parts of the plot of $F$ where the monotonicity of $F$ does not change. By a small abuse of notation, we extend functions $S_i$ by a constant value to the whole of~$\R$. We usually write, for images of functions $S_1$, $S_2$, $S_3$ and for their domains
$$
I_1 = (-\infty,\alpha_{+}], \qquad \qquad I_2 = (\alpha_{+},\beta_{-}), \qquad \qquad I_3 = [\beta_{-}, \infty),
$$
$$
J_1 = (-\infty,f_{+}], \qquad \qquad J_2 = (f_{-},f_{+}), \qquad \qquad J_3 = [f_{-}, \infty) .
$$
\end{Nott}

\begin{Ass}[Initial data for \eqref{eq:fastreact}] \label{ass:ID}
Functions $\ueps(0,x) = u_0(x)$, $\veps(0,x) = v_0(x)$ satisfy 
\begin{enumerate}
\item Nonnegativity: $u_0, v_0 \geq 0$.
\item \label{ass:ic2} Regularity: $u_0, v_0 \in C^{2+\alpha}(\overline{\Omega})$ for some $\alpha \in (0,1)$.
\item Boundary condition: $u_0, v_0$ satisfy the Neumann boundary condition.
\end{enumerate}
\end{Ass}

\begin{Ass}[Reaction function $F$] \label{ass:nonlinearity}
We assume that the function $F(u)$ satisfies:
\begin{enumerate}
\item Regularity,  nonnegativity: $F$ is Lipschitz continuous with $F(0) = 0$ and $F \geq 0$.
\item \label{ass:mon} Piecewise monotonicity of $F$: there are $\alpha_{-}<\alpha_{+}<\beta_{-} < \beta_{+}$ such that $F(\beta_{-}) = F(\alpha_{-})$, $F(\alpha_+) = F(\beta_+)$, $F$ is strictly increasing on $(-\infty,\alpha_+) \cup (\beta_{-}, \infty)$ and strictly decreasing on $(\alpha_{+}, \beta_{-})$ (see Fig.~ \ref{plot:Fd}). Moreover, $\lim_{u\to \infty} F(u) = \infty$.
\item Regularity of inverses: functions $\lambda \mapsto S_i'(\lambda)$ are continuous except $\lambda = f_{-}, f_{+}$. 
\end{enumerate}
\end{Ass}

\section{Properties of the fast-reaction system \eqref{eq:fastreact}}\label{sect:prop_fast_react}
\noindent We begin from recalling energy equality and well-posedness result from \cite{perthame2020fast} which we prove below for the sake of completeness.
\begin{lem}[energy equality] Given a smooth test function $\phi: \R \to \R$, we define
\begin{equation}\label{eq:def_of_Psi_Phi}
\Psi(\lambda) := \int_0^{\lambda} \phi(F(\tau)) \diff \tau, \qquad \qquad 
\Phi(\lambda) := \int_0^{\lambda} \phi(\tau) \diff \tau.
\end{equation}
Then, if $(\ueps, \veps)$ solve \eqref{eq:fastreact}, it holds
\begin{equation}\label{eq:PDE_for_kinetic_function_withTF}
\partial_t \Psi(\ueps) + \partial_t \Phi(\veps) = \Delta \Phi(\veps) - \phi'(\veps)\,|\nabla \veps|^2 - \frac{\big(\veps - F(\ueps)\big)\, (\phi(\veps) - \phi(F(\ueps))\big)}{\varepsilon}.
\end{equation}
\end{lem}
\begin{proof}
Multiplying equation for $\ueps$ in \eqref{eq:fastreact} with $\phi(F(\ueps))$ and equation for $\veps$ in \eqref{eq:fastreact} with $\phi(\veps)$ we obtain
\begin{align*}
\partial_t \Psi(\ueps) &=  \frac{\veps - F(\ueps)}{\varepsilon} \, \phi(F(\ueps)),\\
\partial_t \Phi(\veps) &= \Delta \Phi(\veps) - \phi'(\veps)\,|\nabla \veps|^2 + \frac{F(\ueps) - \veps}{\varepsilon}\, \phi(\veps).
\end{align*}
Summing up these equations we deduce \eqref{eq:PDE_for_kinetic_function_withTF}.
\end{proof}

\begin{lem}\label{theorem:well_posed}
There exists the unique classical solution $\ueps, \veps: [0,\infty)\times \Omega \to \R$  of \eqref{eq:fastreact} which is nonnegative and has regularity
$$
\ueps \in C^{\alpha, 1 + \alpha/2}\left([0,\infty) \times \overline{\Omega}\right), \qquad \veps \in C^{2+\alpha, 1 + \alpha/2}\left([0,\infty) \times \overline{\Omega}\right).
$$
Moreover, we have
\begin{enumerate}
\item \label{PUB0}  $0 \leq \ueps \leq M $, $0 \leq \veps \leq M $ with $M = \max(\|F(u_0)\|_{\infty},\, \|u_0\|_{\infty}, \, \|v_0\|_{\infty},\,f_{+},\,\beta_{+})$,
\item \label{PUB1} $\{\nabla \veps \}_{\varepsilon \in (0,1)}$ is uniformly bounded in $L^2((0,\infty)\times \Omega)$,
\item \label{PUB3} $\left\{\frac{F(\ueps) - \veps}{\sqrt{\varepsilon}}\right\}_{\varepsilon \in (0,1)}$ and $\left\{\sqrt{\varepsilon}\, \Delta \veps \right\}_{\varepsilon \in (0,1)}$ are uniformly bounded in $L^2((0,\infty)\times \Omega)$,
\item \label{PUB5} $\{\partial_t \ueps + \partial_t \veps \}_{\varepsilon \in (0,1)}$ is uniformly bounded in $L^2(0,T;H^{-1}(\Omega))$,
\item \label{PUB6} for all smooth $\varphi: \R \to \R$, $\{\nabla \varphi(\veps) \}_{\varepsilon \in (0,1)}$ is uniformly bounded in $L^2((0,\infty)\times \Omega)$,
\item \label{PUB7} for all smooth $\phi: \R \to \R$, $\{\partial_t \Psi(\ueps) + \partial_t \Phi(\veps) \}_{\varepsilon \in (0,1)}$ is uniformly bounded in $(C(0,T; H^k(\Omega)))^*$ for sufficiently large $k \in \mathbb{N}$.
\end{enumerate}
\end{lem}
\begin{proof}
Existence and uniqueness of global solution as well as points \eqref{PUB0}--\eqref{PUB3} were proven in \cite[Theorem 3.1]{perthame2020fast} so we only sketch the argument. First, local well-posedness and nonnegativity follows from classical theory \cite{MR755878}. To extend existence and uniqueness to an arbitrary interval of time, we need to prove a priori estimates as in \eqref{PUB0}. To this end, we note that thanks to \eqref{eq:PDE_for_kinetic_function_withTF}, the nonnegative map
\begin{equation*}
t \mapsto \int_{\Omega} \big[\Psi(\ueps(t,x)) +\Phi(\veps(t,x))\big] \diff x
\end{equation*}
is nonincreasing whenever $\phi'\geq0$. Choosing $\phi$ vanishing on $(0,M)$ and stricly increasing for $(M,\infty)$ we obtain \eqref{PUB0} and global well-posedness. Then, \eqref{PUB1} and \eqref{PUB3} follows from \eqref{eq:PDE_for_kinetic_function_withTF} with $\phi(v) = v$. Furthemore, \eqref{PUB5} follows from the equality $\partial_t \ueps + \partial_t \veps = \Delta \veps$ and property \eqref{PUB1} while \eqref{PUB6} follows from the chain rule for Sobolev functions, boundedness of $\veps$ from \eqref{PUB0} and \eqref{PUB1}. Finally, to see \eqref{PUB7} we choose $k \geq d$ so that $H^k(\Omega)$ embedds continuously into $L^{\infty}(\Omega)$. Let $\varphi \in C(0,T; H^k(\Omega))$. Note that there is a constant $C$ such that
\begin{equation}\label{eq:embedding_est}
\| \varphi \|_{\infty} \leq C\, \| \varphi \|_{C(0,T; H^k(\Omega))}, \qquad \qquad \| \varphi \|_{L^2(0,T;H^1(\Omega))} \leq C\, \| \varphi \|_{C(0,T; H^k(\Omega))}.
\end{equation}
Thanks to \eqref{eq:PDE_for_kinetic_function_withTF} we have
\begin{multline*}
\int_{(0,T)\times\Omega} \, (\partial_t \Psi(\ueps) + \partial_t \Phi(\veps))\,\varphi \diff t \diff x - \int_{(0,T)\times\Omega} \nabla \Phi(\veps) \cdot \nabla \varphi  \diff t \diff x = \\ =  -\int_{(0,T)\times\Omega} \phi'(\veps)\,|\nabla \veps|^2 \, \varphi \diff t \diff x -\int_{(0,T)\times\Omega} \frac{\big(\veps - F(\ueps)\big)\, (\phi(\veps) - \phi(F(\ueps))\big)}{\varepsilon} \, \varphi \diff t \diff x.
\end{multline*}
As $|\phi'(\veps)|\leq C$ and $|\phi(\veps) - \phi(F(\ueps)|\leq C\, |\veps - F(\ueps)|$ we use bounds \eqref{eq:embedding_est} together with points \eqref{PUB1} and \eqref{PUB3} to deduce for some possibly larger constant $C$ (independent of $\varepsilon$)
$$
\left|\int_{(0,T)\times\Omega} \, (\partial_t \Psi(\ueps) + \partial_t \Phi(\veps))\,\varphi \diff t \diff x \right| \leq C\, \| \varphi \|_{C(0,T; H^k(\Omega))}.
$$
\end{proof}
\begin{cor}\label{cor:connection_YM}
Let $\{\mu_{t,x}\}_{t,x}$ and $\{\nu_{t,x}\}_{t,x}$ be Young measures generated by sequences $\{\ueps\}_{\varepsilon>0}$ and $\{\veps\}_{\varepsilon>0}$ respectively. Combining Lemma \ref{theorem:well_posed} (\ref{PUB3}) and Lemma \ref{lem:YM_prop} (B, C) we obtain that $F^{\#}\mu_{t,x} = \nu_{t,x}$. 
\end{cor}
\section{Proof of Theorem \ref{mainthm:simple_entropy_equality} for fast-reaction system \eqref{eq:fastreact}}\label{sect:proof_mainmain_thm}

\noindent We begin with an entropy equality.
\begin{lem}[entropy equality]\label{lem:entropy_equality}
Let $\Psi$ and $\Phi$ be defined with \eqref{eq:def_of_Psi_Phi}. Let $g_i$ be densities given by \eqref{eq:def_density_gi}. Then, for almost all $\lambda_0$ (with respect to $F^{\#}\mu_{t,x}$) we have
\begin{equation}\label{eq:entropy_equality}
\sum_{i=1}^3 (\Psi(S_i(\lambda_0)) + \Phi(\lambda_0)) \,  g_i(\lambda_0) = \sum_{i=1}^3 \int_{\R^+} (\Psi(S_i(\lambda)) + \Phi(\lambda)) \, g_i(\lambda) \diff F^{\#}\mu_{t,x}(\lambda),
\end{equation}
where $S_i$ are inverses of $F$ as in Notation \ref{intro:not_inv_phi}.
\end{lem}
\begin{proof}
Thanks to Lemma \ref{theorem:well_posed} (\ref{PUB7}), for all smooth $\phi: \R \to \R$, $\{\partial_t \Psi(\ueps) + \partial_t \Phi(\veps) \}_{\varepsilon \in (0,1)}$ is uniformly bounded in $(C(0,T; H^k(\Omega)))^*$. Similarly, for all smooth $\varphi: \R \to \R$, $\{\nabla \varphi(\veps) \}_{\varepsilon \in (0,1)}$ is uniformly bounded in $L^2((0,\infty)\times \Omega)$. Hence, Lemma \ref{lem:comp_comp_moussa} implies
$$
\wstlim_{\varepsilon \to 0} \, (\Psi(\ueps) + \Phi(\veps)) \, \varphi(\veps) = \wstlim_{\varepsilon \to 0} \, (\Psi(\ueps) + \Phi(\veps)) \, \wstlim_{\varepsilon \to 0} \, \varphi(\veps).
$$
As $\veps - F(\ueps) \to 0$ cf. Lemma \ref{theorem:well_posed} (\ref{PUB3}), we may replace $\veps$ with $F(\ueps)$ in the identity above to obtain
$$
\wstlim_{\varepsilon \to 0} \, (\Psi(\ueps) + \Phi(F(\ueps))) \, \varphi(F(\ueps)) = \wstlim_{\varepsilon \to 0} \, (\Psi(\ueps) + \Phi(F(\ueps))) \, \wstlim_{\varepsilon \to 0} \, \varphi(F(\ueps)).
$$
In the language of Young measures, this identity reads
$$
\int_{\R^+} (\Psi(\lambda) + \Phi(F(\lambda))) \,  \varphi(F(\lambda)) \diff \mu_{t,x}(\lambda) = \int_{\R^+} (\Psi(\lambda) + \Phi(F(\lambda)))\diff \mu_{t,x}(\lambda) \, \int_{\R^+} \varphi(F(\lambda)) \diff \mu_{t,x}(\lambda).
$$
We observe that $\lambda = \sum_{i=1}^3 S_i(F(\lambda)) \, \mathds{1}_{\lambda \in I_i}$. Hence, we may use push-forward measure to write
\begin{multline*}
\sum_{i=1}^3 \int_{\R^+} (\Psi(S_i(\lambda)) + \Phi(\lambda)) \,  \varphi(\lambda) \diff F^{\#}\mu_{t,x}^{(i)}(\lambda) = \\ = \sum_{i=1}^3 \int_{\R^+} (\Psi(S_i(\lambda)) + \Phi(\lambda)) \diff F^{\#}\mu_{t,x}^{(i)}(\lambda) \, \int_{\R^+} \varphi(\lambda) \diff F^{\#} \mu_{t,x}(\lambda).
\end{multline*}
Using \eqref{eq:def_density_gi} with densities $g_1(\lambda)$, $g_2(\lambda)$ and $g_3(\lambda)$ we obtain
\begin{multline*}
\sum_{i=1}^3 \int_{\R^+} (\Psi(S_i(\lambda)) + \Phi(\lambda)) \,  \varphi(\lambda) \, g_i(\lambda) \diff F^{\#}\mu_{t,x}(\lambda) = \\ = \sum_{i=1}^3 \int_{\R^+} (\Psi(S_i(\lambda)) + \Phi(\lambda)) \, g_i(\lambda) \diff F^{\#}\mu_{t,x}(\lambda) \, \int_{\R^+} \varphi(\lambda) \diff F^{\#} \mu_{t,x}(\lambda).
\end{multline*}
Hence, when $\lambda_0$ belongs to the support of the measure $F^{\#} \mu_{t,x}$, we obtain
$$
\sum_{i=1}^3 (\Psi(S_i(\lambda_0)) + \Phi(\lambda_0)) \,  g_i(\lambda_0) = \sum_{i=1}^3 \int_{\R^+} (\Psi(S_i(\lambda)) + \Phi(\lambda)) \, g_i(\lambda) \diff F^{\#}\mu_{t,x}(\lambda).
$$
\end{proof}
\noindent To analyze entropy inequality, we need to deal with integrals of the form $\int_{0}^{S_i(\lambda)} \phi(F(\tau)) \diff \tau$. This is the content of the next lemma.
\begin{lem}\label{lem:change_of_variables_complicated}
We have
$$
\Psi(S_i(\lambda_0)) = \int_{0}^{S_i(\lambda_0)} \phi(F(\tau)) \diff \tau = \int_{0}^{\lambda_0} \phi(\tau) \, S_{i}'(\tau) \diff \tau + C_i(\phi)
$$
where $C_1(\phi) = 0$ and $C_2(\phi) = C_3(\phi) = \int_{0}^{f_{+}} \phi(\tau) \, \left(S_1'(\tau) - S_2'(\tau)\right) \diff \tau$.
\end{lem}
\begin{proof}
For $i = 1$ we note that $F$ is invertible on $(0, S_1(\lambda))$ so that simple change of variables implies
$$
\Psi(S_1(\lambda_0)) = \int_{0}^{S_1(\lambda_0)} \phi(F(\tau)) \diff \tau = \int_{0}^{\lambda_0} \phi(\tau) \, S_{1}'(\tau) \diff \tau.
$$
For $i=2$ we first split the integral for two intervals $(0, \alpha_+),\, (\alpha_+, \lambda_0)$ cf. Notation \ref{intro:not_inv_phi}. On each of them $F$ is invertible so we can apply change of variables again:
\begin{multline*}
\Psi(S_2(\lambda_0)) = \int_{0}^{\alpha_{+}} \phi(F(\tau)) \diff \tau + \int_{\alpha_{+}}^{S_2(\lambda_0)}  \phi(F(\tau)) \diff \tau = \\ = 
\int_{0}^{f_{+}} \phi(\tau) \, S_1'(\tau) \diff \tau - \int_{\lambda_0}^{f_{+}}  \phi(\tau) \, S_2'(\tau) \diff \tau 
= C_2(\phi) + \int_{0}^{\lambda_0}  \phi(\tau) \, S_2'(\tau) \diff \tau. 
\end{multline*}
For $i=3$ we split the integral for three intervals and apply change of variables again:
\begin{multline*}
\Psi(S_3(\lambda_0)) = \int_{0}^{\alpha_{+}} \phi(F(\tau)) \diff \tau + \int_{\alpha_{+}}^{\beta_-}  \phi(F(\tau)) \diff \tau + \int_{\beta_{-}}^{S_3(\lambda_0)}  \phi(F(\tau)) \diff \tau= \\ = 
\int_{0}^{f_{+}} \phi(\tau) \, S_1'(\tau) \diff \tau - \int_{f_{-}}^{f_{+}}  \phi(\tau) \, S_2'(\tau) \diff \tau 
+  \int_{f_{-}}^{\lambda_0}  \phi(\tau) \, S_3'(\tau) \diff \tau.
\end{multline*}
As $S_2'(\tau) = 0$ and $S_3'(\tau) = 0$ for $\tau \in (0,f_{-})$, the proof is concluded.
\end{proof}
\begin{lem}\label{thm:useful_entropy_identity}
Consider function
$$
\mathcal{F}(\tau_0) = \sum_{i=1}^3  (S_{i}'(\tau_0) + 1) \, F^{\#}\mu_{t,x}^{(i)}((\tau_0, \infty))  +  (S_{1}'(\tau_0) - S_{2}'(\tau_0)) \, (1-  F^{\#}\mu_{t,x}^{(1)}(\R^+) ).
$$
Then, for almost all $\lambda_0$ (with respect to $F^{\#}\mu_{t,x}$) and $\tau_0 \neq f_{-},f_{+}$ we have
$$
\mathds{1}_{\lambda_0 > \tau_0} \, \sum_{i=1}^3  (S_{i}'(\tau_0) + 1)\, g_i(\lambda_0)  +  (S_{1}'(\tau_0) - S_{2}'(\tau_0)) \, (1-g_1(\lambda_0)) = \mathcal{F}(\tau_0).
$$
\end{lem}
\begin{proof}
We consider $\phi(\tau) = \phi^{\delta}(\tau) = \frac{1}{ \delta} \, \mathds{1}_{[\tau_0, \tau_0+\delta]}$ and send $\delta \to 0$ so that $\Phi(\lambda_0)=\int_0^{\lambda_0}\phi^{\delta}(\tau) \diff \tau \to \mathds{1}_{\lambda > \tau_0}$. Moreover, $\int_0^{\lambda_0} \phi^{\delta}(\tau) \, S_i'(\tau) \diff \tau \to S_i'(\tau_0)\, \mathds{1}_{\lambda_0 > \tau_0}$. Therefore, from Lemmas \ref{lem:entropy_equality} and \ref{lem:change_of_variables_complicated} we deduce
\begin{multline*}
\sum_{i=1}^3 \Big(\mathds{1}_{\lambda_0 > \tau_0} \, (S_{i}'(\tau_0) + 1)  +  (S_{1}'(\tau_0) - S_{2}'(\tau_0)) \, \mathds{1}_{i=2,3} \Big) \,  g_i(\lambda_0) = \\ =
\sum_{i=1}^3 \int_{\R^+} \Big(\mathds{1}_{\lambda > \tau_0} \, (S_{i}'(\tau_0) + 1)  +  (S_{1}'(\tau_0) - S_{2}'(\tau_0)) \, \mathds{1}_{i=2,3}\Big) \, g_i(\lambda)  \diff F^{\#}\mu_{t,x}(\lambda).
\end{multline*}
Using identities from \eqref{eq:summabilityFmu} and \eqref{eq:summabilityg}
$$
1-g_1(\lambda_0) = g_2(\lambda_0) + g_3(\lambda_0), \qquad 1-  F^{\#}\mu_{t,x}^{(1)}(\R^+) =   F^{\#}\mu_{t,x}^{(2)}(\R^+) + F^{\#}\mu_{t,x}^{(3)}(\R^+)
$$
we conclude the proof.
\end{proof}

\begin{proof}[Proof of Theorem \ref{mainthm:simple_entropy_equality}]
\noindent The first part of Theorem \ref{mainthm:simple_entropy_equality} is proved in Lemma \ref{thm:useful_entropy_identity}. To see the second one, fix $\lambda_0 \neq f_{-}, f_{+}$. For $\tau_0:= \eta > \lambda_0$ we obtain
$$
\sum_{i=1}^3  (S_{i}'(\eta) + 1) \, F^{\#}\mu_{t,x}^{(i)}((\eta, \infty)) + (S_1'(\eta) - S_2'(\eta)) \, (F^{\#}\mu_{t,x}^{(1)}(\R^+) - g_1(\lambda_0)) = 0
$$
while for $\tau_0:= \xi < \lambda_0$ we deduce
$$
 \sum_{i=1}^3  (S_{i}'(\xi) + 1)\, (g_i(\lambda_0) - F^{\#}\mu_{t,x}^{(i)}((\xi, \infty))) + (S_1'(\xi) - S_2'(\xi)) \, (F^{\#}\mu_{t,x}^{(1)}(\R^+) - g_1(\lambda_0))= 0.
$$
Sending $\xi, \eta \to \lambda_0$ and using continuity of $\lambda \mapsto S_i'(\lambda)$ at $\lambda \neq f_{-}, f_{+}$ we obtain
$$
 \sum_{i=1}^3  (S_{i}'(\lambda_0) + 1)\, g_i(\lambda_0) = \sum_{i=1}^3  (S_{i}'(\lambda_0) + 1) \, F^{\#}\mu_{t,x}^{(i)}\{\lambda_0\}.
$$
Finally, we note that for almost all $\lambda_0$ (with respect to $F^{\#}\mu_{t,x}$) $F^{\#}\mu_{t,x}^{(i)}\{\lambda_0\} = g_i(\lambda_0)\, F^{\#}\mu_{t,x}\{\lambda_0\}$ and this concludes the proof.
\end{proof}

\section{Proofs of Theorems \ref{mainthm:oldplotnikov}, \ref{mainthm:emptyunstablephase} and \ref{mainthm:newconditiononF} for fast-reaction system \eqref{eq:fastreact}}\label{sect:proofs_theorems}
\begin{proof}[Proof of Theorem \ref{mainthm:oldplotnikov}]
If $\suppo F^{\#}\mu_{t,x} \cap (0,f_{-})$ is nonempty, we take any $\lambda_0 \in \suppo F^{\#}\mu_{t,x} \cap (0,f_{-})$. Note that $S_2'(\lambda_0) = S_3'(\lambda_0) = 0$ and \eqref{eq:localized_final_entropy} in Theorem \ref{mainthm:simple_entropy_equality} implies
$$
\left(1 - F^{\#}\mu_{t,x}\left\{\lambda_0\right\} \right) \,  (S_1'(\lambda_0) + 1) \, g_1(\lambda_0) \,  = 0.
$$
For almost all $\lambda_0 \in (0,f_{-})$ we have $g_1(\lambda_0) = 1$ so we conclude $F^{\#}\mu_{t,x}\left\{\lambda_0\right\} = 1$. Similar argument works in the case $\lambda_0 \in (f_{+}, \infty)$. \\

\noindent Now, let $\lambda_0 \in [f_{-}, f_{+}] \cap \suppo F^{\#}\mu_{t,x}$. If $\suppo F^{\#}\mu_{t,x} = \{\lambda_0\}$, we conclude $F^{\#}\mu_{t,x} = \delta_{\lambda_0}$. Otherwise, there are $\lambda_1, \lambda_2 \in \suppo F^{\#}\mu_{t,x}$ such that $f_{-} \leq \lambda_1 <\lambda_2 \leq f_{+}$. For any $\tau_0 \in (\lambda_1, \lambda_2)$ we use Theorem~\ref{mainthm:simple_entropy_equality} with $\lambda_0 = \lambda_1, \lambda_2$ to obtain two equations:
$$
\sum_{i=1}^3 (S_i'(\tau_0) + 1) \, \left[ \, g_i(\lambda_2) - F^{\#}\mu_{t,x}^{(i)}(\tau_0, \infty) \right] + (S_1'(\tau_0) - S_2'(\tau_0)) \, (F^{\#}\mu_{t,x}^{(1)}(\R^+) - g_1(\lambda_2))= 0,
$$
$$
-\sum_{i=1}^3 (S_i'(\tau_0) + 1) \, F^{\#}\mu_{t,x}^{(i)}(\tau_0, \infty) + (S_1'(\tau_0) - S_2'(\tau_0)) \, (F^{\#}\mu_{t,x}^{(1)}(\R^+) - g_1(\lambda_1)) = 0.
$$
Hence, $\sum_{i=1}^3 (S_i'(\tau_0) + 1) \,  g_i(\lambda_2) + (S_1'(\tau_0) - S_2'(\tau_0)) \, (g_1(\lambda_1) - g_1(\lambda_2)) = 0$. But then, nondegeneracy condition \eqref{eq:nondegeneracy:fast_reaction} implies that the sum $\sum_{i=1}^3  g_i(\lambda_2) = 0 \neq 1$ raising contradiction.\\

\noindent It follows that $F^{\#}\mu_{t,x}$ is a Dirac mass. From Corollary \ref{cor:connection_YM} we deduce that the Young measure $\{\nu_{t,x}\}_{t,x}$ generated by $\{\veps\}_{\varepsilon}$ is also a Dirac mass so $\veps \to v$ strongly and $\nu_{t,x} = \delta_{v(t,x)}$, cf. Lemma \ref{lem:YM_prop}. The representation formula for $\mu_{t,x}$ follows from $F^{\#}\mu_{t,x} = \delta_{v(t,x)}$.
\end{proof}

\noindent Before proceeding to the proofs of Theorems \ref{mainthm:emptyunstablephase} and \ref{mainthm:newconditiononF}, we will state a simple lemma concerning the case when $F^{\#}\mu_{t,x}$ is supported only at $f_{-}$ and $f_{+}$. This needs some care as functions $S_1'$, $S_2'$ and $S_3'$ are not continuous at these points.
\begin{lem}[Accumulation at the interface]\label{lem:case_f-f+}
Suppose that there exists $\tau_0 \in (f_-, f_+)$ such that $S_1'(\tau_0) - S_3'(\tau_0) \neq 0$.
Assume that $\suppo F^{\#}\mu_{t,x} \subset \{f_{-},f_{+}\}$. Then, $F^{\#}\mu_{t,x} = \delta_{f_{-}}$ or $F^{\#}\mu_{t,x} = \delta_{f_{+}}$.
\end{lem}
\begin{proof}
Aiming at contradiction, we assume that $F^{\#}\mu_{t,x}\{f_{+}\} > 0$ and $F^{\#}\mu_{t,x}\{f_{-}\} > 0$. Note that
$$
0 = \mu_{t,x}^{(2)}(F^{-1}(f_{+}) \cap I_2) = F^{\#}\mu_{t,x}^{(2)}\{f_{+}\} = g_2(f_{+}) \, F^{\#}\mu_{t,x}\{f_{+}\}
$$
so that $g_2(f_{+}) = 0$ and similarly $g_2(f_{-}) = 0$. Applying Theorem \ref{mainthm:simple_entropy_equality} with $\tau_0 \in (f_{-},f_{+})$ and $\lambda_0~\in~\{f_-, f_{+}\}$ we obtain
$$
\sum_{i=1}^3 (S_i'(\tau_0) + 1) \, \left[\mathds{1}_{\lambda_0 > \tau_0} \, g_i(\lambda_0) - F^{\#}\mu_{t,x}^{(i)}(\tau_0, \infty) \right] + (S_1'(\tau_0) - S_2'(\tau_0)) \, (F^{\#}\mu_{t,x}^{(1)}(\R^+) - g_1(\lambda_0))= 0.
$$
As $\tau_0 \in (f_{-},f_{+})$, we have
$$
F^{\#}\mu_{t,x}^{(i)}(\tau_0, \infty) = F^{\#}\mu_{t,x}^{(i)}\{f_+\} = g_i(f_{+}) \, F^{\#}\mu_{t,x}\{f_{+}\}.
$$
But this implies
$$
\left(\mathds{1}_{\lambda_0>\tau_0} - F^{\#}\mu_{t,x}\{f_{+}\} \right) \, \sum_{i=1, 3} (S_i'(\tau_0) + 1) \,g_i(\lambda_0) + (S_1'(\tau_0) - S_2'(\tau_0)) \, (F^{\#}\mu_{t,x}^{(1)}(\R^+) - g_1(\lambda_0))  = 0.
$$
Considering $\lambda_0 = f_{+}, \,f_{-}$ and using $1-F^{\#}\mu_{t,x}\{f_{+}\}= F^{\#}\mu_{t,x}\{f_{-}\}$ we obtain two equations:
\begin{equation}\label{eq:fplus}
F^{\#}\mu_{t,x}\{f_{-}\} \, \sum_{i=1, 3} (S_i'(\tau_0) + 1) \,g_i(f_+) + (S_1'(\tau_0) - S_2'(\tau_0)) \, (F^{\#}\mu_{t,x}^{(1)}(\R^+) - g_1(f_+))  = 0,
\end{equation}
\begin{equation}\label{eq:fminus}
-  F^{\#}\mu_{t,x}\{f_{+}\} \, \sum_{i=1, 3} (S_i'(\tau_0) + 1) \,g_i(f_-) + (S_1'(\tau_0) - S_2'(\tau_0)) \, (F^{\#}\mu_{t,x}^{(1)}(\R^+) - g_1(f_-))  = 0.
\end{equation}
Using $1-F^{\#}\mu_{t,x}\{f_{+}\}= F^{\#}\mu_{t,x}\{f_{-}\}$ once again we obtain that 
\begin{multline*}
F^{\#}\mu_{t,x}^{(1)}(\R^+) - g_1(f_+) = 
g_1(f_+) \, F^{\#}\mu_{t,x}\{f_{+}\} + g_1(f_-) \, F^{\#}\mu_{t,x}\{f_{-}\} - g_1(f_+) = \\ = (g_1(f_-) - g_1(f_+)) \, F^{\#}\mu_{t,x}\{f_{-}\}
\end{multline*}
and similarly for $F^{\#}\mu_{t,x}^{(1)}(\R^+) - g_1(f_-)$. As we assume that $F^{\#}\mu_{t,x}\{f_{-}\}, F^{\#}\mu_{t,x}\{f_{+}\} > 0$ we may simplify \eqref{eq:fplus}--\eqref{eq:fminus} to obtain
\begin{equation}\label{eq:fplus_mod1}
\sum_{i=1, 3} (S_i'(\tau_0) + 1) \,g_i(f_+) + (S_1'(\tau_0) - S_2'(\tau_0)) \, (g_1(f_-) - g_1(f_+))  = 0,
\end{equation}
\begin{equation}\label{eq:fminus_mod1}
-  \sum_{i=1, 3} (S_i'(\tau_0) + 1) \,g_i(f_-) + (S_1'(\tau_0) - S_2'(\tau_0)) \, (g_1(f_+) - g_1(f_-))  = 0.
\end{equation}
We observe further that $g_1(\lambda_0) + g_3(\lambda_0) = 1$, cf. \eqref{eq:summabilityg}, so that
$$
\sum_{i=1, 3} (S_i'(\tau_0) + 1) \,g_i(\lambda_0) = (S_1'(\tau_0) - S_3'(\tau_0)) \, g_1(\lambda_0) + (S_3'(\tau_0) + 1).
$$
Hence, we may further simplify \eqref{eq:fplus_mod1}--\eqref{eq:fminus_mod1} to get
\begin{equation}\label{eq:fplus_mod2}
(S_1'(\tau_0) - S_3'(\tau_0)) \, g_1(f_+) + (S_3'(\tau_0) + 1) + (S_1'(\tau_0) - S_2'(\tau_0)) \, (g_1(f_-) - g_1(f_+))  = 0,
\end{equation}
\begin{equation}\label{eq:fminus_mod2}
-  (S_1'(\tau_0) - S_3'(\tau_0)) \, g_1(f_-) - (S_3'(\tau_0) + 1) + (S_1'(\tau_0) - S_2'(\tau_0)) \, (g_1(f_+) - g_1(f_-))  = 0.
\end{equation}
By assumption, there is $\tau_0 \in (f_-, f_+)$ such that $S_1'(\tau_0) - S_3'(\tau_0) \neq 0$. Using \eqref{eq:fplus_mod2}--\eqref{eq:fminus_mod2} for such $\tau_0$ we see that $g_1(f_+) = g_1(f_-)$. But then, coming back to \eqref{eq:fplus_mod2}--\eqref{eq:fminus_mod2}, we deduce that 
$$
\sum_{i=1, 3} (S_i'(\tau_0) + 1) \,g_i(f_-) = 0, \qquad \sum_{i=1, 3} (S_i'(\tau_0) + 1) \,g_i(f_-) = 0.
$$
As $S_1$, $S_3$ are increasing, this implies $g_1(f_-) = g_3(f_-) = g_1(f_+) = g_3(f_+) = 0$ raising contradiction with $g_1(f_-) + g_3(f_-) = 1$ and $g_1(f_+) + g_3(f_+) = 1$. 
\end{proof}
\begin{rem}\label{rem:sharper_ass_casef-f+}
Without assumption that there is $\tau_0 \in (f_-, f_+)$ such that $S_1'(\tau_0) - S_3'(\tau_0) \neq 0$ we observe that \eqref{eq:fplus_mod2}--\eqref{eq:fminus_mod2} degenerate to the same equation:
$$
g_1(f_+) - g_1(f_-) = \frac{1+S_3'(\tau_0)}{S_1'(\tau_0) - S_2'(\tau_0)}
$$
valid for all $\tau_0 \in (f_-, f_+)$. Hence, it the function $\tau_0 \mapsto \frac{1+S_3'(\tau_0)}{S_1'(\tau_0) - S_2'(\tau_0)}$ is not constant, we may also obtain contradiction. But we believe that assumption on $S_1'(\tau_0) - S_3'(\tau_0)$ is easier to formulate. It also allows for piecewise affine nonlinearities.
\end{rem}

\begin{proof}[Proof of Theorem \ref{mainthm:emptyunstablephase}]
As in the proof of Theorem \ref{mainthm:oldplotnikov} we may assume that $\suppo F^{\#}\mu_{t,x} \subset [f_{-}, f_{+}]$ (this did not use nondegeneracy condition!). By assumption, for any set $A \subset \R^+$
$$
0 = \mu_{t,x}(F^{-1}(A) \cap I_2) = F^{\#}\mu_{t,x}^{(2)}(A) = \int_{A} g_2(\lambda) \diff F^{\#}\mu_{t,x}(\lambda)
$$
so $g_2(\lambda) = 0$ for almost all $\lambda$. Hence, when $\lambda_0 \in \suppo F^{\#}{\mu_{t,x}} \cap (f_{-}, f_{+})$, the sum
$$
\sum_{i=1}^3 (S_i'(\lambda_0) + 1) \, g_i(\lambda_0)  \geq \mbox{min}(S_1'(\lambda_0) + 1, S_3'(\lambda_0) + 1) > 0
$$
because $g_1(\lambda_0) + g_3(\lambda_0) = 1$ and $S_1$, $S_3$ are strictly increasing. It follows from Theorem \ref{mainthm:simple_entropy_equality} that $F^{\#}\mu_{t,x}\{\lambda_0\}=1$, i.e. $F^{\#}\mu_{t,x} = \delta_{\lambda_0}$. Finally, if there is no such $\lambda_0 \in \suppo F^{\#}{\mu_{t,x}} \cap (f_{-}, f_{+})$, we apply Lemma \ref{lem:case_f-f+}.\\

\noindent It follows that $F^{\#}\mu_{t,x}$ is a Dirac mass and now, we can conclude as in Theorem \ref{mainthm:oldplotnikov}.
\end{proof}
\begin{proof}[Proof of Theorem \ref{mainthm:newconditiononF}]
Mimicking the proof of Theorem \ref{mainthm:emptyunstablephase}, we let $\lambda_0 \in \suppo F^{\#}{\mu_{t,x}} \cap (f_{-}, f_{+})$ and we observe that the sum 
$$
\sum_{i=1}^3 (S_i'(\lambda_0) + 1)\,  g_i(\lambda_0)  \geq \mbox{min}(1, \delta(\lambda_0)) \, \sum_{i=1}^3 g_i(\lambda_0) = \mbox{min}(1, \delta(\lambda_0)) > 0 
$$
where $\delta(\lambda_0)$ is such that $S_2'(\lambda_0) + 1 > \delta(\lambda_0)>0$. We conclude as in the proof of Theorem \ref{mainthm:emptyunstablephase}. 
\end{proof}
\section{Proof of Theorems \ref{mainthm:simple_entropy_equality}--\ref{mainthm:emptyunstablephase} to the forward-backward diffusion system \eqref{eq:forback}}\label{sect:adaptation}

\noindent We first formulate basic well-posedness result for \eqref{eq:forback}. This comes mostly from \cite{MR1015926,MR1299852} but the compactness estimates are simplified.

\begin{lem}\label{theorem:well_posed:forback}
Let $u_0 \in L^{\infty}(\Omega)$. Then, there exists the unique solution $\ueps:[0,\infty)\times \Omega \to \R$ of \eqref{eq:forback} which is nonnegative and has regularity
$
C^1([0,T]; L^{2}(\Omega)) \cap L^{\infty}(\Omega).
$
Moreover, we have
\begin{enumerate}
\item \label{PUB0_fb} for $M = \max(\|F(u_0)\|_{\infty},\, f_{+})$ we have $0 \leq \ueps \leq M $,
\item \label{PUB1_fb} $\{\nabla \veps \}_{\varepsilon \in (0,1)}$ is uniformly bounded in $L^2((0,T)\times \Omega)$,
\item \label{PUB3_fb} $\left\{\frac{\veps - F(\ueps)}{\sqrt{\varepsilon}}\right\}_{\varepsilon \in (0,1)} = \left\{\sqrt{\varepsilon}\, \ueps_t \ \right\}_{\varepsilon \in (0,1)}$ are uniformly bounded in $L^2((0,T)\times \Omega)$,
\item \label{PUB6_fb} for all smooth $\varphi: \R \to \R$, $\{\nabla \varphi(\veps) \}_{\varepsilon \in (0,1)}$ is uniformly bounded in $L^2((0,T)\times \Omega)$,
\item \label{PUB7_fb} for all smooth $\phi: \R \to \R$, $\{\partial_t \Psi(\ueps) \}_{\varepsilon \in (0,1)}$ is uniformly bounded in $(C(0,T; H^k(\Omega)))^*$ for sufficiently large $k \in \mathbb{N}$.
\end{enumerate}
\end{lem}
\begin{proof}
We observe that equation is equivalent to the following ODE:
$$
\partial_t \ueps = (I - \varepsilon\, \Delta)^{-1} \, \Delta \, F(\ueps).
$$
As long as $\varepsilon > 0$, the (RHS) is Lipschitz continuous, say on $L^{2}(\Omega)$,  so the local well-posedness follows. To obtain global well-posedness, we consider functions $\Psi$, $\Phi$ defined in \eqref{eq:def_of_Psi_Phi}. We have
\begin{equation}\label{eq:entropy_time_der_forback}
\begin{split}
\partial_t \Psi(\ueps) = \phi(F(\ueps)) \, \ueps_t = (\phi(F(\ueps)) &- \phi(\veps)) \, \ueps_t + \phi(\veps) \, \Delta \veps = \\ &=
(\phi(F(\ueps)) - \phi(\veps)) \, \ueps_t + \Delta \Phi(\veps)  - \phi'(\veps) \, |\nabla \veps|^2.
\end{split}
\end{equation}
If $\phi$ is nondecreasing, we have 
$$
(\phi(F(\ueps)) - \phi(\veps)) \, \ueps_t = (\phi(F(\ueps)) - \phi(\veps)) \, \frac{\veps - F(\ueps)}{\varepsilon} \leq 0
$$
so after integration in space, the (RHS) of \eqref{eq:entropy_time_der_forback} is nonnegative. Hence, $\partial_t \int_{\Omega} \Psi(\ueps) \leq 0$. Choosing $\phi = 0$ for $[0,M]$ and $\phi'(x)>0$ for $x \notin [0,M]$ we prove \eqref{PUB0_fb} and conclude the proof of global well-posedness. To see \eqref{PUB1_fb} and \eqref{PUB3_fb} we take $\phi(x) = x$ and integrate \eqref{eq:entropy_time_der_forback} in time and space. Part \eqref{PUB6_fb} easily follows from chain rule and \eqref{PUB1_fb}. Finally, \eqref{PUB7_fb} follows from \eqref{eq:entropy_time_der_forback} and exactly the same computations as in Lemma \ref{theorem:well_posed}.
\end{proof}

\noindent Now, we formulate an analog of Lemma \ref{lem:entropy_equality}.
\begin{lem}[entropy equality]\label{lem:entropy_for_forback}
Let $\Psi$ be defined with \eqref{eq:def_of_Psi_Phi}. Let $g_i$ be densities given by \eqref{eq:def_density_gi}. Then, for almost all $\lambda_0$ (with respect to $F^{\#}\mu_{t,x}$) we have
\begin{equation}\label{eq:entropy_equality:forback}
\sum_{i=1}^3 \Psi(S_i(\lambda_0)) \,  g_i(\lambda_0) = \sum_{i=1}^3 \int_{\R^+} \Psi(S_i(\lambda))  \, g_i(\lambda) \diff F^{\#}\mu_{t,x}(\lambda).
\end{equation}
\end{lem}

\begin{proof}
Thanks to Lemma \ref{theorem:well_posed:forback} \eqref{PUB7_fb}, for all smooth $\phi: \R \to \R$, $\{\partial_t \Psi(\ueps)\}_{\varepsilon \in (0,1)}$ is uniformly bounded in $(C(0,T; H^k(\Omega)))^*$. Similarly, for all smooth and bounded $\varphi: \R \to \R$, $\{\nabla \varphi(\veps) \}_{\varepsilon \in (0,1)}$ is uniformly bounded in $L^2((0,\infty)\times \Omega)$. Hence, Lemma \ref{lem:comp_comp_moussa} implies
$$
\wstlim_{\varepsilon \to 0} \, \Psi(\ueps)  \, \varphi(\veps) = \wstlim_{\varepsilon \to 0} \, \Psi(\ueps)  \, \wstlim_{\varepsilon \to 0} \, \varphi(\veps).
$$
As $\veps - F(\ueps) = \varepsilon \, \ueps_t \to 0$ cf. Lemma \ref{theorem:well_posed:forback} \eqref{PUB3_fb}, we may replace $\veps$ with $F(\ueps)$ in the identity above to obtain
$$
\wstlim_{\varepsilon \to 0} \, \Psi(\ueps) \, \varphi(F(\ueps))) = \wstlim_{\varepsilon \to 0} \, \Psi(\ueps)  \, \wstlim_{\varepsilon \to 0} \, \varphi(F(\ueps))).
$$
In the language of Young measures, this identity reads
$$
\int_{\R^+} \Psi(\lambda)  \,  \varphi(F(\lambda)) \diff \mu_{t,x}(\lambda) = \int_{\R^+} \Psi(\lambda) \diff \mu_{t,x}(\lambda) \, \int_{\R^+} \varphi(F(\lambda)) \diff \mu_{t,x}(\lambda).
$$
We observe that $\lambda = \sum_{i=1}^3 S_i(F(\lambda)) \, \mathds{1}_{\lambda \in I_i}$. Hence, we may use push-forward measure to write
\begin{equation*}
\sum_{i=1}^3 \int_{\R^+} \Psi(S_i(\lambda))  \,  \varphi(\lambda) \diff F^{\#}\mu_{t,x}^{(i)}(\lambda) =  \sum_{i=1}^3 \int_{\R^+} \Psi(S_i(\lambda))  \diff F^{\#}\mu_{t,x}^{(i)}(\lambda) \, \int_{\R^+} \varphi(\lambda) \diff F^{\#} \mu_{t,x}(\lambda).
\end{equation*}
Using densities $g_1(\lambda)$, $g_2(\lambda)$ and $g_3(\lambda)$ we obtain
\begin{multline*}
\sum_{i=1}^3 \int_{\R^+} \Psi(S_i(\lambda)) \,  \varphi(\lambda) \, g_i(\lambda) \diff F^{\#}\mu_{t,x}(\lambda) = \\ = \sum_{i=1}^3 \int_{\R^+} \Psi(S_i(\lambda)) \, g_i(\lambda) \diff F^{\#}\mu_{t,x}(\lambda) \, \int_{\R^+} \varphi(\lambda) \diff F^{\#} \mu_{t,x}(\lambda).
\end{multline*}
Hence, if $\lambda_0$ belongs to the support of the measure $F^{\#} \mu_{t,x}$, we obtain \eqref{eq:entropy_equality:forback}.
\end{proof}

\begin{proof}[Proof of Theorems \ref{mainthm:simple_entropy_equality}--\ref{mainthm:emptyunstablephase}] Comparing formulations of Lemma \ref{lem:entropy_equality} and \ref{lem:entropy_for_forback} we see that it is sufficient to modify proofs in Sections \ref{sect:proof_mainmain_thm}-\ref{sect:proofs_theorems} by replacing $S_1' + 1$, $S_2' + 1$ and $S_3' + 1$ with $S_1'$, $S_2'$ and $S_3'$ respectively. 
\end{proof}

\noindent Note that Theorem \ref{mainthm:newconditiononF} is only true for fast-reaction limit \eqref{eq:fastreact} because its proof exploits presence of $S_2' + 1$ in the entropy formulations.
\appendix
\section{Useful notions and results}\label{app:genlabel}
\subsection{Compensated compactness lemma} We formulate lemma used in the proof of Theorem \ref{mainthm:simple_entropy_equality}, more precisely in Lemma \ref{lem:entropy_equality}. For the proof see \cite[Proposition 1]{MR3466213}.
\begin{lem}\label{lem:comp_comp_moussa}
Let $\Omega \subset \R^n$ be a bounded domain. Suppose that $\{a_n\}_{n \in \N}$ is uniformly bounded in $L^2(0,T; H^1(\Omega))$ and $\{b_n\}_{n\in \N}$ is uniformly bounded in $L^2(0,T; L^2(\Omega))$. Moreover, assume that the sequence of distributional time derivatives $\{\partial_t b_n \}_{n \in \N}$ is uniformly bounded in the dual space $C(0,T; H^m(\Omega))^*$ for some $m \in \N$. Then, if $a_n \weak a$ and $b_n \weak b$ we have $a_n \, b_n \to a \, b$ in the sense of distributions.
\end{lem}
\noindent In our case, the considered sequences are also in $L^{\infty}((0,T)\times \Omega)$ so the resulting convergence is true in the weak$^*$ sense.

\subsection{Support of the measure} We recall definition of the support of measure on $\R^n$ \cite[Definition 1.14]{MR3409718}. For this, let $B(x,r)$ denote a ball of radius $r>0$ centered at $x \in \R^n$.

\begin{Def}
Let $\mu$ be a nonnegative measure on $\R^n$. We say that $x \in \suppo \mu$ if and only if $\mu(B(x,r)) > 0$ for all $r>0$.
\end{Def}

\begin{rem}
When given property (like equation) is satisfied for almost every $x$ (with respect to $\mu$) one may worry that it is not true for the particularly chosen value of $x$. This is not the problem if one takes $x \in \suppo \mu$ because in each neighbourhood of $x$ there is $y \in \suppo \mu$ such that the property has to be satisfied because the measure of each neighbourhood is nonzero. 
\end{rem}
\subsection{Young measures}\label{app:YM} Finally, we recall the theory of Young measures introduced by Young \cite{MR6832,MR6023} and recalled in the seminal paper of Ball \cite{MR1036070}. Reader interested in modern presentation may consult \cite{MR1034481}, \cite[Chapter 6]{MR1452107} or \cite[Chapter 4]{MR3821514}. For simplicity, we formulate it for sequences of functions $\{u_n\}_{n \in \N}$ uniformly bounded in $L^{p}(\Omega)$ with some $1 \leq p \leq \infty$ and $\Omega \subset \R^n$ being a bounded domain. 
\noindent We start with the most important result that we cite from \cite[Theorem 6.2]{MR1452107}:
\begin{thm}[Fundamental Theorem of Young Measures]\label{app:fundamental_thm}
Let $\Omega \subset \R^n$ be a a bounded domain and let $\{u_n\}_{n \in \mathbb{N}}$ be a sequence bounded in $L^{p}(\Omega)$ with $1 \leq p \leq \infty$. Then, there exists a subsequence (not relabeled) and a weakly-$\ast$ measurable family of probability measures $\{\mu_x\}_{x \in \Omega}$ such that for all bounded and smooth $G: \R \to \R$, we have
\begin{equation}\label{app:weaklimitL1}
G(u_n(x)) \weaks \int_{\R} G(\lambda) \, \diff \mu_x(\lambda) \qquad \mbox{ in } L^{\infty}(\Omega).
\end{equation}
We say that the sequence $\{u_n\}_{n \in \N}$ generates the family of Young measures $\{\mu_x\}_{x\in \Omega}$. 
\end{thm}
\noindent Now we list properties of Young measures used in the paper.
\begin{lem}\label{lem:YM_prop} Under notation of Theorem \ref{app:fundamental_thm}, the following hold true.
\begin{itemize}
    \item[(A)] We have 
$
u_n \to u 
$
a.e. (up to a subsequence) if and only if $ \mu_{t,x} = \delta_{u(t,x)}$.
    \item[(B)] If $\{w_n\}_{n\in \N}$ is another bounded sequence such that $u_n - w_n \to 0$ a.e. then Young measures generated by $\{u_n\}_{n\in \N}$ and $\{w_n\}_{n \in \N}$ coincide.
    \item[(C)] If $F: \R \to \R$, the sequence $\{F(u_n)\}_{n\in \N}$ generates Young measure $F^{\#}\mu_{t,x}$ (i.e. push-forward $\mu_{t,x} \circ F^{-1}$). 
\end{itemize}
\end{lem}
\begin{proof}[Sketch of the proof]
For (A) we consider $G(u) = u$ and $G(u) = u^2$ to deduce that $u_n \to u$ in $L^2(\Omega)$ so that up to a subsequence also a.e. The opposite direction is clear because $G(u_n(x)) \to G(u(x))$ a.e. For (B) we note that for all bounded and smooth $G$, weak limits of $G(u_n(x))$ and $G(w_n(x))$ coincide. For (C) we write
$$
G(F(u_n)) = \int_{\R} G(F(\lambda)) \diff \mu_{t,x}(\lambda) = \int_{\R} G(\lambda) \diff (\mu_{t,x} \circ F^{-1})(\lambda).
$$
\end{proof}
\bibliographystyle{abbrv}
\bibliography{fastlimit}

\end{document}